\documentclass[final]{siamltex}

\usepackage[left=3cm, right=3cm, top=3cm, bottom=3cm]{geometry}

\usepackage{hyperref}
\usepackage{url}
\usepackage{amsmath,amssymb,bm,bbm}
\usepackage{amsfonts,dsfont}
\usepackage{graphics}
\usepackage{epsfig}
\usepackage{tablefootnote}
\usepackage{cite}

\newtheorem{thm}{Theorem}[section]  

\newtheorem{lem}[thm]{Lemma}

\newtheorem{corr}[thm]{Corollary}

\newtheorem{asmp}[thm]{Assumption}

\newcommand{\norm}[1]{\left|\left|#1\right|\right|}
\newcommand{\nn}{\nonumber}
\newcommand{\Real}{{\mathbbm{R}}}

\newcommand{\la}{\langle}
\newcommand{\ra}{\rangle}

\newcommand{\order}{\mathcal{O}}
\newcommand{\torder}{\widetilde{\mathcal{O}}}

\DeclareMathOperator*{\argmin}{arg\,min}

\newcommand{\prox}{\text{prox}}

\newcommand{\sj}{\sum_{j=(k-K)_+}^{k-1}}

\newcommand{\nfk}{\nabla f(x_k)}

\newcommand{\nfs}{\nabla f(x^*)}
\newcommand{\ndk}{\norm{d_k}}
\newcommand{\ndj}{\norm{d_j}}
\newcommand{\ndi}{\norm{d_i}}

\newcommand{\nxkxs}{\norm{x_k-x^*}}

\title{Global Convergence Rate of Proximal Incremental Aggregated Gradient Methods}

\author{
N.~D.~Vanli\thanks{Laboratory for Information and Decision Systems, Massachusetts Institute of Technology, Cambridge, MA 02139, USA. email: \{denizcan, mertg, asuman\}@mit.edu.}
\and M.~G\"urb\"uzbalaban\footnotemark[1]
\and A.~Ozdaglar\footnotemark[1]
}
\date{\today}

\begin{document}

\maketitle

\begin{abstract}
We focus on the problem of minimizing the sum of smooth component functions (where the sum is strongly convex) and a non-smooth convex function, which arises in regularized empirical risk minimization in machine learning and distributed constrained optimization in wireless sensor networks and smart grids. We consider solving this problem using the proximal incremental aggregated gradient (PIAG) method, which at each iteration moves along an aggregated gradient (formed by incrementally updating gradients of component functions according to a deterministic order) and taking a proximal step with respect to the non-smooth function. While the convergence properties of this method with randomized orders (in updating gradients of component functions) have been investigated, this paper, to the best of our knowledge, is the first study that establishes the convergence rate properties of the PIAG method for any deterministic order. In particular, we show that the PIAG algorithm is globally convergent with a linear rate provided that the step size is sufficiently small. We explicitly identify the rate of convergence and the corresponding step size to achieve this convergence rate. Our results improve upon the best known condition number dependence of the convergence rate of the incremental aggregated gradient methods used for minimizing a sum of smooth functions.
\end{abstract}

\section{Introduction}\label{sec:intro}
We focus on {\it composite additive cost optimization problems}, where the objective function is given by the sum of $m$ component functions $f_i(x)$ and a possibly non-smooth regularization function $r(x)$:
\begin{equation}\label{eq:goal}
  \min_{x\in\Real^n} F(x) \triangleq f(x) + r(x),
\end{equation}
and $f(x) = \frac{1}{m} \sum_{i=1}^m f_i(x)$. We assume each component function $f_i:\Real^n\to(-\infty,\infty)$ is convex and continuously differentiable while the regularization function $r:\Real^n\to(-\infty,\infty]$ is proper, closed, and convex but not necessarily differentiable.
This formulation arises in many problems in machine learning, distributed optimization, and signal processing. Notable examples include constrained and regularized least squares problems that arise in various machine learning applications \cite{BachSagaMethod14,hogwild}, distributed optimization problems that arise in wireless sensor network as well as smart grid applications \cite{nedic2001distributed, smartgrid} and constrained optimization of separable problems \cite{bertsekas2011incremental}.
An important feature of this formulation is that the number of component functions $m$ is large, hence solving this problem using a standard gradient method that involves evaluating the full gradient of $f(x)$, i.e., $\nabla f(x) = \sum_{i=1}^m \nabla f_i(x)$, is costly. This motivates using {\it incremental methods} that exploit the additive structure of the problem and update the decision vector using one component function at a time.

When $r$ is continuously differentiable, one widely studied approach is the incremental gradient (IG) method \cite{NedicBertsekas01,Solodov98IncrGrad, bertsekas2011incremental}. The IG method processes the component functions one at a time by taking steps along the gradient of each individual function in a sequential manner, following a cyclic order \cite{TsengYun2014incremental,tsitsiklis_bertsekas} or a randomized order \cite{RobbinsMonro,mert_RandomReshuffling,tsitsiklis_bertsekas}. A particular randomized order, which at each iteration independently picks a component function uniformly at random from all component functions leads to the popular stochastic gradient descent (SGD) method. While SGD is the method of choice in practice for many machine learning applications due to its superior empirical performance and convergence rate estimates that does not depend on the number of component functions $m$, its convergence rate is  sublinear, i.e., an $\epsilon$-optimal solution can be computed within $ O(1/\epsilon)$ iterations.\footnote{Let $x^*$ be an optimal solution of the problem \eqref{eq:goal}. A vector $x\in\Real^n$ is an {\em $\epsilon$-optimal solution} if $F(x)-F(x^*)\leq\epsilon$.} In a seminal paper, Blatt {\it et al.}  \cite{Blatt2007incremental} proposed the {\it incremental aggregated gradient (IAG) method}, which maintains the savings associated with incrementally accessing the component functions, but keeps the most recent component gradients in memory to approximate the full gradient $\nabla f(x)$ and updates the iterate using this aggregated gradient. Blatt {\it et al.} showed that under some assumptions, for a sufficiently small constant step size, the IAG method is globally convergent and when the component functions are quadratics, it achieves a linear rate. Two recent papers, \cite{Leroux2012sgd} and \cite{mert_iag}, investigated the convergence rate of this method for general component functions that are convex and smooth (i.e., with Lipschitz gradients), where the sum of the component functions is strongly convex: In \cite{Leroux2012sgd}, the authors focused on a randomized version, called stochastic average gradient (SAG) method (which samples the component functions independently similar to SGD), and showed that it achieves a linear rate using a proof that relies on the stochastic nature of the algorithm. In a more recent work \cite{mert_iag}, the authors focused on deterministic IAG (i.e., component functions processed using an arbitrary deterministic order) and provided a simple analysis that uses a delayed dynamical system approach to study the evolution of the iterates generated by this algorithm.

While these recent advances suggest IAG as a promising approach with fast convergence rate guarantees for solving additive cost problems, in many applications listed above, the objective function takes a composite form and includes a non-smooth regularization function $r(x)$ (to avoid overfitting or to induce a sparse representation). Another important case of interest is smooth constrained optimization problems which can be represented in the composite form (\ref{eq:goal}) where the function $r(x)$ is the indicator function of a nonempty closed convex set.

In this paper, we study the {\it proximal incremental aggregated gradient (PIAG) method} for solving composite additive cost optimization problems. Our method computes an aggregated gradient for the function $f(x)$ (with component gradients evaluated in a {\it deterministic manner} at outdated iterates over a finite window $K$, similar to IAG) and uses a proximal operator with respect to the regularization function $r(x)$ at the intermediate iterate obtained by moving along the aggregated gradient. Under the assumptions that $f(x)$ is strongly convex and each $f_i(x)$ is smooth with Lipschitz gradients, we show the first {\it linear convergence rate} result for the deterministic PIAG and provide explicit convergence rate estimates that highlight the dependence on the condition number of the problem (which we denote by $Q$) and the size of the window $K$ over which outdated component gradients are evaluated. In particular, we show that in order to achieve an $\epsilon$-optimal solution, the PIAG algorithm requires $\order(QK^2 \log^2(QK)\log(1/\epsilon))$ iterations, or equivalently $\torder(QK^2\log(1/\epsilon))$ iterations, where the tilde is used to hide the logarithmic terms in $Q$ and $K$. This result improves upon the condition number dependence of the deterministic IAG for smooth problems \cite{mert_iag}, where the authors proved that to achieve an $\epsilon$-optimal solution, the IAG algorithm requires $\order(Q^2K^2\log(1/\epsilon))$ iterations. We also note that two recent independent papers \cite{proxgrad_schmidt, proxgrad_lewis} have analyzed the convergence rate of the prox-gradient algorithm (which is a special case of our algorithm with $K=0$, i.e., where we have access to a full gradient at each iteration instead of an aggregated gradient) under strong convexity type assumptions and provided linear rate estimates. Our rate estimates for the PIAG algorithm with $K>0$ matches the condition number dependence of the prox-gradient algorithm provided in these papers \cite{proxgrad_schmidt, proxgrad_lewis} up to logarithmic factors. Furthermore, for the case $K=0$ (i.e., for the prox-gradient algorithm), the rate estimates obtained using our analysis technique can be shown to have the same condition number dependence as the ones presented in \cite{proxgrad_schmidt, proxgrad_lewis}.

Our analysis uses function values to track the evolution of the iterates generated by the PIAG algorithm. This is in contrast with the recent analysis of the IAG algorithm provided in \cite{mert_iag}, which used distances of the iterates to the optimal solution as a {\em Lyapunov function} and relied on the smoothness of the problem to bound the gradient errors with distances. This approach does not extend to the non-smooth composite case, which motivates a new analysis using function values and the properties of the proximal operator. Since we work directly with function values, this approach also allows us to obtain iteration complexity results to achieve an $\epsilon$-optimal solution.


In terms of the algorithmic structure, our paper is related to \cite{BachSagaMethod14}, where the authors introduce the SAGA method, which extends the SAG method to the composite case and provides a linear convergence rate result with an analysis that relies on the stochastic nature of the algorithm and does not extend to the deterministic case. Particularly, the SAGA method samples the component functions randomly and independently at each iteration without replacement (in contrast with the PIAG method, where the component functions are processed deterministically). However, such random sampling may not be possible for applications such as decentralized information processing in wireless sensor networks (where agents are subject to communication constraints imposed by the network topology and all agents are not necessarily connected to every other agent via a low-cost link \cite{nedic2001distributed}), motivating the study of the deterministic PIAG method. In \cite{BachSagaMethod14}, the authors prove that to achieve a point in the $\epsilon$-neighborhood of the optimal solution, SAGA requires $\order(\max(Q,K)\log(1/\epsilon))$ iterations.\footnote{Let $x^*$ be an optimal solution of the problem \eqref{eq:goal}. A vector $x\in\Real^n$ is in the {\em $\epsilon$-neighborhood of an optimal solution} if $\norm{x-x^*}\leq\epsilon$.} However, note that this result does not translate into a guarantee in the function suboptimality of the resulting point because of lack of smoothness. Furthermore, the choice of Lyapunov function in \cite{BachSagaMethod14} requires each $f_i(x)$ to be convex (to satisfy the non-negativity condition), whereas we do not need this assumption in our analysis.

Our work is also related to \cite{TsengYun2014incremental}, where the authors propose a related linearly convergent incrementally updated gradient method for solving the composite additive cost problem in \eqref{eq:goal} under a local Lipschitzian error condition (a condition satisfied by locally strongly convex functions around an optimal solution). The PIAG algorithm is different from the algorithm proposed in \cite{TsengYun2014incremental}. Specifically, for constrained optimization problems (i.e., when the regularization function is the indicator function of a nonempty closed convex set), the iterates generated by the algorithm in \cite{TsengYun2014incremental} stay in the interior of the set since the algorithm in \cite{TsengYun2014incremental} searches for a feasible update direction. On the other hand, the PIAG algorithm uses the proximal map on the intermediate iterate obtained by moving in the opposite direction of the aggregated gradient, which operates as a projected gradient method and allows the iterates to be on the boundary of the set. Aside from algorithmic differences, \cite{TsengYun2014incremental} does not provide explicit rate estimates (even though the exact rate can be calculated after an elaborate analysis, the dependence on the condition number and the window length of the outdated gradients is significantly worse than the one presented in this paper). Furthermore, the results in \cite{TsengYun2014incremental} provides a $K$-step linear convergence, whereas the linear convergence results in our paper hold uniformly for each step.

Other than the papers mentioned above, our paper is also related to  \cite{bertsekas_iap}, which studies an alternative incremental aggregated proximal method and shows linear convergence when each $f_i(x)$ and $r(x)$ are continuously differentiable. This method forms a linear approximation to $f(x)$ and processes the component functions $f_i(x)$ with a proximal iteration, whereas our method processes $f_i(x)$ based on a gradient step. Furthermore, our linear convergence results do not require the differentiability of the objective function $r(x)$ in contrast to the analysis in \cite{bertsekas_iap}.

Several recent papers in the machine learning literature (e.g., \cite{BachSagaMethod14,svrg,finito,miso,universal_catalyst} and references therein) are also weakly related to our paper. In all these papers, the authors propose randomized order algorithms similar to the SAG algorithm \cite{Leroux2012sgd} and analyze their convergence rates in expectation. In particular, in \cite{finito}, the authors propose an algorithm, called Finito, which is closely related to the SAG algorithm but achieves a faster convergence rate than the SAG algorithm. These ideas are then extended to composite optimization problems with non-smooth objective functions (as in \eqref{eq:goal}) in \cite{miso,BachSagaMethod14}. In particular, in \cite{miso}, a majorization-minimization algorithm, called MISO, is proposed to solve smooth optimization problems and its global linear convergence is shown in expectation. In \cite{universal_catalyst}, the ideas in \cite{miso} are then extended for non-smooth optimization problems using proximal operator. Similarly, in \cite{svrg}, a variance reduction technique is applied to the SGD algorithm for smooth problems and its global linear convergence in expectation is proven.

The rest of the paper is organized as follows. In Section \ref{sec:piag}, we introduce the PIAG algorithm. In Section \ref{sec:main}, we first provide the assumptions on the objective functions and then prove the global linear convergence of the proposed algorithm under these assumptions. We conclude the paper in Section \ref{sec:conclusion} with a summary of our results.

\section{The PIAG Algorithm}\label{sec:piag}
Similar to the IAG method, at each iteration $k$, we form an aggregated gradient, which we denote as follows
\begin{equation}
  g_k \triangleq \frac{1}{m} \sum_{i=1}^m \nabla f_i(x_{\tau_{i,k}}), \nn
\end{equation}
where $\nabla f_i(x_{\tau_{i,k}})$ represents the gradient of the $i$th component function sampled at time $\tau_{i,k}$. We assume that each component function is sampled at least once in the past $K\geq0$ iterations, i.e., we have
\begin{equation}
  k-K \leq \tau_{i,k} \leq k, \,\,\, \forall i\in\{1,\dots,m\}. \nn
\end{equation}
This condition is typically satisfied in practical implementations of the deterministic incremental methods. For instance, if the functions are processed in a cyclic order, we have $K=m-1$  \cite{TsengYun2014incremental, mert_RandomReshuffling}. On the other hand, $K=0$ corresponds to the case where we have the full gradient of the function $f(x)$ at each iteration (i.e., $g_k=\nfk$) and small $K$ may represent a setting in which the gradients of the component functions are sent to a processor with some delay upper bounded by $K$.

Since the regularization function $r$ is not necessarily differentiable, we propose to solve \eqref{eq:goal} with the proximal incremental aggregated gradient (PIAG) method, which uses the proximal operator with respect to the regularization function at the intermediate iterate obtained using the aggregated gradient. In particular, the PIAG algorithm, at each iteration $k\geq0$, updates $x_k$ as
\begin{equation}\label{eq:updaterule}
  x_{k+1} = \prox_r^\eta(x_k - \eta g_k),
\end{equation}
where $\eta$ is a constant step size and the proximal mapping is defined as follows
\begin{equation}\label{eq:prox}
  \prox_r^\eta(y) = \argmin_{x\in\Real^n} \left\{ \frac{1}{2} \norm{x-y}^2 + \eta r(x) \right\}.
\end{equation}
Here, we define $\phi(x) \triangleq \frac{1}{2} \norm{x-y}^2 + \eta r(x)$ and let $\partial \phi(x)$ denote the set of subgradients of the function $\phi$ at $x$. Then, it follows from the optimality conditions \cite{Bertsekas15Book} of the problem in \eqref{eq:prox} that $0 \in \partial \phi(x_{k+1})$. This yields $x_{k+1}-(x_k - \eta g_k)+\eta h_{k+1}=0$, for some $h_{k+1} \in \partial r(x_{k+1})$. Hence, we can compactly represent our update rule as
\begin{equation}\label{eq:update}
  x_{k+1} = x_k + \eta d_k,
\end{equation}
where $d_k \triangleq -g_k-h_{k+1}$ is the direction of the update at time $k$.

\section{Convergence Analysis}\label{sec:main}
\subsection{Assumptions}
Throughout the paper, we make the following standard assumptions.

\begin{asmp}\label{asmp:lips}\textbf{(Lipschitz gradients)}
Each $f_i$ has Lipschitz continuous gradients on $\Real^n$ with some constant $L_i \geq 0$, i.e.,
\begin{equation}
  \norm{\nabla f_i(x)-\nabla f_i(y)} \leq L_i \norm{x-y}, \nn
\end{equation}
for any $x,y\in\Real^n$.\footnote{If a function $f$ has Lipschitz continuous gradients with some constant $L$, then $f$ is called $L$-smooth. We use these terms interchangeably.}
\end{asmp}

Defining $L \triangleq \frac{1}{m} \sum_{i=1}^m L_i$, we observe that Assumption \ref{asmp:lips} and the triangle inequality yield
\begin{equation}
  \norm{\nabla f(x)-\nabla f(y)} \leq L \norm{x-y}, \nn
\end{equation}
for any $x,y\in\Real^n$, i.e., the function $f$ is $L$-smooth.

\begin{asmp}\label{asmp:conv}\textbf{(Strong Convexity)}
The sum function $f$ is $\mu$-strongly convex on $\Real^n$ for some $\mu>0$, i.e., the function $x \mapsto f(x) - \frac{\mu}{2} \norm{x}^2$ is convex, and the regularization function $r$ is convex on $(-\infty,\infty]$.
\end{asmp}

A consequence of Assumption \ref{asmp:conv} is that $F$ is strongly convex, hence there exists a unique optimal solution of problem \eqref{eq:goal} \cite[Lemma 6]{primaldual}, which we denote by $x^*$.

We emphasize that these assumptions hold for a variety of cost functions including regularized squared error loss, hinge loss, and logistic loss \cite{Bubeck15CvxBook} and similar assumptions are widely used to analyze the convergence properties of incremental gradient methods in the literature \cite{BachSagaMethod14,bertsekas_iap,bertsekas_ipm,Leroux2012sgd,mert_iag}. Note that in contrast with many of these analyses, we do not assume that the component functions $f_i$ are convex.

\subsection{Rate of Convergence}
In this section, we show that the PIAG algorithm attains a global linear convergence rate with a constant step size provided that the step size is sufficiently small. We define
\begin{equation}\label{eq:lyapunov}
  F_k \triangleq F(x_k) - F(x^*),
\end{equation}
which is the suboptimality in the objective value at iteration $k$. In our analysis, we will use $F_k$ as a Lyapunov function to prove global linear convergence. Before providing the main theorems of the paper, we first introduce three lemmas that contain key relations in proving these theorems.

The first lemma investigates how the suboptimality in the objective value evolves over the iterations. In particular, it shows that the change in suboptimality $F_{k+1} - F_k$ can be bounded as a sum of two terms: The first term is negative and has a linear dependence in the step size $\eta$, whereas the second term is positive and has a quadratic dependence in $\eta$. This suggests that if the step size $\eta$ is small enough, the linear term in $\eta$ will be dominant guaranteeing a descent in suboptimality.

\begin{lem}\label{prop1}
Suppose that Assumptions 1 and 2 hold. Then, the PIAG algorithm in \eqref{eq:updaterule} yields the following guarantee
\begin{equation}
  F_{k+1} \leq F_k - \frac{1}{2} \eta \norm{d_k}^2 + \eta^2 \frac{L}{2} \sj\ndj^2, \nn
\end{equation}
for any step size $0<\eta\leq\frac{1}{L(K+1)}$.
\end{lem}

\begin{proof}
We first consider the difference of the errors in consecutive time instances and write
\begin{align}
  F(x_{k+1}) - F(x_k) & = f(x_{k+1}) - f(x_k) + r(x_{k+1}) - r(x_k) \nn\\
    & \leq \la \nfk, x_{k+1}-x_k \ra + \frac{L}{2} \norm{x_{k+1}-x_k}^2  + r(x_{k+1}) - r(x_k), \nn
\end{align}
where the inequality follows from the Taylor series expansion of $f$ around $x_k$ and since the Hessian of $f$ at any point is upper bounded by $L$ by Assumption \ref{asmp:lips}. Using the update rule $x_{k+1}=x_k+\eta d_k$ in this inequality, we obtain
\begin{align}
  F(x_{k+1}) - F(x_k) & = \eta \la \nfk, d_k \ra + \eta^2 \frac{L}{2} \ndk^2 + r(x_{k+1}) - r(x_k) \nn\\
    & = \eta \la \nfk-g_k, d_k \ra + \eta^2 \frac{L}{2} \ndk^2 + \eta \la g_k, d_k \ra + r(x_{k+1}) - r(x_k) \nn\\
    & \leq \eta \norm{\nfk-g_k} \norm{d_k} + \eta^2 \frac{L}{2} \ndk^2 - \eta \ndk^2 - \eta \la h_{k+1}, d_k \ra + r(x_{k+1}) - r(x_k) \nn\\
    & = \eta \norm{\nfk-g_k} \norm{d_k} + \eta \left( \eta \frac{L}{2} - 1 \right) \ndk^2 + \la h_{k+1}, x_k - x_{k+1} \ra + r(x_{k+1}) - r(x_k) \nn\\
    & \leq \eta \norm{\nfk-g_k} \norm{d_k} + \eta \left( \eta \frac{L}{2} - 1 \right) \ndk^2, \label{eq:c1}
\end{align}
where the first inequality follows by the triangle inequality and the last inequality follows from the convexity of $r$.

The gradient error term in \eqref{eq:c1}, i.e., $\norm{\nfk-g_k}$, can be upper bounded as follows
\begin{align}
  \norm{\nfk-g_k} & \leq \frac{1}{m} \sum_{i=1}^m \norm{\nabla f_i(x_k) - \nabla f_i(x_{\tau_{i,k}})} \nn\\
    & \leq \frac{1}{m} \sum_{i=1}^m L_i \norm{x_k - x_{\tau_{i,k}}} \nn\\
    & \leq \frac{1}{m} \sum_{i=1}^m L_i \sum_{j=\tau_{i,k}}^{k-1} \eta \ndj \nn\\
    & \leq \eta L \sj \ndj, \label{eq:lem1}
\end{align}
where the first and third inequalities follow by the triangle inequality, the second inequality follows since each $f_i$ is $L_i$-smooth, and the last inequality follows since $\tau_{i,k} \geq k-K$. Using \eqref{eq:lem1} we can upper bound \eqref{eq:c1} as follows
\begin{align}
  F(x_{k+1}) - F(x_k) & \leq \eta \left( \eta \frac{L}{2} - 1 \right) \ndk^2 +  \eta^2 L \sj \ndj \ndk \nn\\
    & \leq \eta \left( \eta \frac{L(K+1)}{2} - 1 \right) \ndk^2 + \eta^2 \frac{L}{2} \sj \ndj^2 \nn\\
    & \leq - \frac{\eta}{2} \ndk^2 + \eta^2 \frac{L}{2} \sj \ndj^2, \label{eq:c2}
\end{align}
where the second inequality follows from the arithmetic-geometric mean inequality, i.e., $\ndj \ndk \leq \frac{1}{2} (\ndj^2 + \ndk^2)$ and the last inequality follows since $0<\eta\leq\frac{1}{L(K+1)}$. This concludes the proof of Lemma \ref{prop1}.
\end{proof}

We next introduce the following lemma, which can be viewed as an extension of \cite[Theorem 4]{TseYunBlockDescent09} into our framework with aggregated gradients. We provide a simplified proof compared to \cite{TseYunBlockDescent09} with a tighter upper bound. This lemma can be interpreted as follows. When the regularization function is zero (i.e., $r(x)=0$ for all $x\in\Real^n$) and we have access to full gradients (i.e., $K=0$), this lemma simply follows from the strong convexity of the sum function $f$ since $\nxkxs\leq\frac{1}{\mu}\norm{\nfk-\nfs}$ and $\nfs=0$ due to the optimality condition of the problem. The following lemma indicates that even though we do not have such control over the subgradients of the regularization function (as the regularization function is neither strongly convex nor smooth), the properties of the proximal step yields a similar relation at the expense of a constant of $2$ (instead of $1$ compared to the $r(x)=0$ case) and certain history dependent terms (that arise due to the incremental nature of the PIAG algorithm) that has a linear dependence in step size $\eta$. This lemma will be a key step in the proof of Lemma \ref{prop2}, where we illustrate how the descent term in Lemma \ref{prop1} relates to our Lyapunov function.

\begin{lem}\label{lem}
Suppose that Assumptions 1 and 2 hold and let $Q=L/\mu$ denote the condition number of the problem. Then, the distance of the iterates from the optimal solution is upper bounded as
\begin{equation}
  \nxkxs \leq \frac{2}{\mu} \norm{d_k} + 2 \eta Q \sj \ndj, \nn
\end{equation}
for any $k\geq0$ and $0<\eta\leq\frac{1}{L}$.
\end{lem}

\begin{proof}
Define
\begin{equation}
  d_k' \triangleq \argmin_{d\in\Real^n} \left\{ \frac{\eta}{2} \norm{\nfk+d}^2 + r(x_k+\eta d) \right\}, \nn
\end{equation}
as the direction of update with the full gradient. Non-expansiveness property of the proximal map implies
\begin{equation}
  \norm{\prox_r^\eta(x)-\prox_r^\eta(y)}^2 \leq \la \prox_r^\eta(x)-\prox_r^\eta(y), x-y\ra. \nn
\end{equation}
Putting $x=x_k-\eta\nfk$ and $y=x^*-\eta\nfs$ in the above inequality, we obtain
\begin{align}
  \norm{x_k+\eta d_k'-x^*}^2 & \leq \la x_k+\eta d_k'-x^*, \, x_k-\eta\nfk-x^*+\eta\nfs\ra \nn\\
    & \leq \la x_k+\eta d_k'-x^*, \, x_k+\eta d_k'-x^*\ra + \la x_k+\eta d_k'-x^*, \, -\eta d_k'+\eta\nfs-\eta\nfk\ra, \nn
\end{align}
which implies
\begin{equation}
  0 \leq \la x_k+\eta d_k'-x^*, \, -d_k'+\nfs-\nfk \ra. \nn
\end{equation}
This inequality can be rewritten as follows
\begin{align}
  \la x_k-x^*, \, \nfk-\nfs \ra & \leq \la x_k-x^*, \, -d_k' \ra - \eta \norm{d_k'}^2 + \eta \la d_k', \, \nfs-\nfk \ra \nn\\
    & \leq \la x_k-x^*, \, -d_k' \ra + \eta \la d_k', \, \nfs-\nfk \ra \nn\\
    & \leq \norm{d_k'} \left( \norm{x_k-x^*} + \eta \norm{\nfs-\nfk} \right) \nn\\
    & \leq \norm{d_k'} \left( \norm{x_k-x^*} + \eta L \norm{x_k-x^*} \right) \nn\\
    & \leq 2 \norm{d_k'} \norm{x_k-x^*}, \label{eq:app1}
\end{align}
where the second inequality follows since $-\norm{d_k'}^2\leq0$, the third inequality follows by the Cauchy-Schwarz inequality, the fourth inequality follows from the $L$-smoothness of $f$, and the last inequality follows since $\eta\leq\frac{1}{L}$. Since $\mu$-strong convexity of $f$ implies
\begin{equation}
  \mu\norm{x_k-x^*}^2 \leq \la x_k-x^*, \, \nfk-\nfs \ra, \label{eq:app2}
\end{equation}
then combining \eqref{eq:app1} and \eqref{eq:app2}, we obtain
\begin{equation}
  \mu\norm{x_k-x^*} \leq 2 \norm{d_k'}. \label{eq:app3}
\end{equation}
In order to relate $d_k'$ to the original direction of update $d_k$, we use the triangle inequality and write
\begin{align}
  \norm{d_k'} & \leq \norm{d_k} + \norm{d_k'-d_k} \nn\\
    & = \norm{d_k} + \frac{1}{\eta} \norm{x_k+\eta d_k'-x_k-\eta d_k} \nn\\
    & = \norm{d_k} + \frac{1}{\eta} \norm{\prox_r^\eta(x_k-\eta\nfk) - \prox_r^\eta(x_k-\eta g_k)} \nn\\
    & \leq \norm{d_k} + \norm{g_k-\nfk}, \nn\\
    & \leq \norm{d_k} + \eta L \sj \ndj, \label{eq:app4}
\end{align}
where the last line follows by equation \eqref{eq:lem1}. Putting \eqref{eq:app4} back into \eqref{eq:app3} concludes the proof of Lemma \ref{lem}.
\end{proof}

In the following lemma, we relate the direction of update to the suboptimality in the objective value at a given iteration $k$. In particular, we show that the descent term presented in Lemma \ref{prop1} (i.e., $-\ndk^2$) can be upper bounded by the negative of the suboptimality in the objective value of the next iteration (i.e., $-F_{k+1}$) and additional history dependent terms that arise due to the incremental nature of the PIAG algorithm.

\begin{lem}\label{prop2}
Suppose that Assumptions 1 and 2 hold. Then, for any $0<\eta\leq\frac{1}{L(K+1)}$, the PIAG algorithm in \eqref{eq:updaterule} yields the following guarantee
\begin{equation}
  -\ndk^2 \leq - \frac{\mu}{4} F_{k+1} + \eta L \sj \ndj^2. \nn
\end{equation}
\end{lem}

\begin{proof}
In order to prove this lemma, we use Lemma \ref{lem}, which can be rewritten as follows
\begin{equation}
  - \norm{d_k} \leq - \frac{\mu}{2} \nxkxs  + \eta L \sj \ndj. \nn
\end{equation}
Then, we can upper bound $-\ndk^2$ as
\begin{align}
  - \ndk^2 & \leq - \frac{\mu}{2} \ndk \nxkxs  + \eta L \sj \ndk \ndj \nn\\
    & \leq - \frac{\mu}{2} \la d_k, x^*-x_k \ra + \eta \frac{KL}{2} \ndk^2  + \eta \frac{L}{2} \sj \ndj^2, \label{eq:pop1}
\end{align}
where the last line follows by the Cauchy-Schwarz inequality and the arithmetic-geometric mean inequality. We can upper bound the inner product term in \eqref{eq:pop1} as
\begin{align}
  - \la d_k, x^*-x_k \ra & = \la g_k + h_{k+1}, x^*-x_k \ra  \nn\\
    & = \la \nfk, x^*-x_k \ra + \la h_{k+1}, x^*-x_k \ra + \la g_k - \nfk, x^*-x_k \ra  \nn\\
    & \leq f(x^*) - f(x_k) + \la h_{k+1}, x^*-x_{k+1} \ra + \eta \la h_{k+1}, d_k \ra + \la g_k - \nfk, x^*-x_k \ra  \nn\\
    & \leq f(x^*) - f(x_k) + r(x^*) - r(x_{k+1}) + \eta \la h_{k+1}, d_k \ra + \norm{g_k - \nfk} \norm{x^*-x_k},  \label{eq:pop2}
\end{align}
where the first inequality follows from the convexity of $f$ and the second inequality follows from the convexity of $r$ and the triangle inequality. The inner product term in \eqref{eq:pop2} can be upper bounded as
\begin{align}
  \eta \la h_{k+1}, d_k \ra & = -\eta \ndk^2 - \la g_k, \eta d_k \ra  \nn\\
    & = -\eta \ndk^2 + \la \nfk, - \eta d_k \ra + \la g_k - \nfk, -\eta d_k \ra  \nn\\
    & \leq -\eta \ndk^2 + \la \nfk, x_k-x_{k+1} \ra + \eta \ndk \norm{g_k - \nfk}  \nn\\
    & \leq -\eta \ndk^2 + f(x_k) - f(x_{k+1}) + \eta^2 \frac{L}{2} \ndk^2 + \eta \ndk \norm{g_k - \nfk},  \label{eq:pop3}
\end{align}
where the first inequality follows by the triangle inequality and the second inequality follows from the $L$-smoothness of $f$. Putting \eqref{eq:pop3} back in \eqref{eq:pop2}, we obtain
\begin{equation}
  - \la d_k, x^*-x_k \ra \leq -F_{k+1} + \eta \left( \eta \frac{L}{2} - 1 \right) \ndk^2 + \norm{g_k - \nfk} \left( \norm{x^*-x_k} + \eta \ndk \right). \label{eq:pop4}
\end{equation}
The final term in \eqref{eq:pop4} can be upper bounded as follows
\begin{align}
  \norm{g_k - \nfk} \left( \norm{x^*-x_k} + \eta \ndk \right) & \leq \eta L \left( \sj \ndj \right) \left( \norm{x^*-x_k} + \eta \ndk \right)  \nn\\
    & \leq \eta L \left( \sj \ndj \right) \left[ \left( \eta + \frac{2}{\mu} \right) \norm{d_k} + 2 \eta Q \sj \ndj \right], \nn
\end{align}
where the first line follows by equation \eqref{eq:lem1} and the last line follows by Lemma \ref{lem}. Using arithmetic-geometric mean inequality in the above inequality, we obtain
\begin{equation}
  \norm{g_k - \nfk} \left( \norm{x^*-x_k} + \eta \ndk \right) \leq \eta \frac{KL}{2} \left( \eta + \frac{2}{\mu} \right) \norm{d_k} + \eta \left[ \eta \frac{L}{2} + Q + 2 \eta KQL \right] \sj \ndj. \label{eq:pop5}
\end{equation}
Putting \eqref{eq:pop5} back in \eqref{eq:pop4} yields
\begin{align}
  - \la d_k, x^*-x_k \ra & \leq - F_{k+1} + \eta \left( \eta \frac{L}{2} - 1 + \frac{KL}{2} \left( \eta + \frac{2}{\mu} \right) \right) \ndk^2 + \eta \left[ \eta \frac{L}{2} + Q + 2 \eta KQL \right] \sj \ndj \nn\\
    & = - F_{k+1} + \eta \left( \eta \frac{(K+1)L}{2} - 1 + KQ \right) \ndk^2 + \eta \left[ \eta \frac{L}{2} + Q + 2 \eta KQL \right] \sj \ndj \nn\\
    & \leq - F_{k+1} + \eta \left( KQ - \frac{1}{2} \right) \ndk^2 + \eta \left( \eta \frac{L}{2} + Q + 2 \eta KQL \right) \sj \ndj, \label{eq:pop6}
\end{align}
where the last line follows since $\eta\leq\frac{1}{L(K+1)}$. Finally, using \eqref{eq:pop6} in our original inequality in \eqref{eq:pop1}, we obtain
\begin{align}
  - \ndk^2 & \leq - \frac{\mu}{2} F_{k+1} + \eta \left( \frac{KL}{2} - \frac{\mu}{4} + \frac{KL}{2} \right) \ndk^2 + \eta \left( \eta \frac{\mu L}{4} + \frac{L}{2} + \eta KL^2 + \frac{L}{2} \right) \sj \ndj^2, \nn\\
    & \leq - \frac{\mu}{2} F_{k+1} + \eta KL \ndk^2 + \eta L \left( \frac{\mu}{4} + \eta KL + 1 \right) \sj \ndj^2, \nn\\
    & \leq - \frac{\mu}{2} F_{k+1} + \eta KL \ndk^2 + \eta L \left( \eta (K+1)L + 1 \right) \sj \ndj^2, \nn\\
    & \leq - \frac{\mu}{2} F_{k+1} + \ndk^2 + 2 \eta L \sj \ndj^2, \label{eq:pop7}
\end{align}
where the second inequality follows since $\mu\geq0$, the third inequality follows since $\frac{\mu}{4}\leq L$, and the last inequality follows since $\eta\leq\frac{1}{L(K+1)}$. Rearranging the terms in \eqref{eq:pop7}, we obtain
\begin{equation}
  - \ndk^2 \leq - \frac{\mu}{4} F_{k+1} + \eta L \sj \ndj^2,
\end{equation}
which completes the proof of Lemma \ref{prop2}.
\end{proof}

In the following theorem, we derive a recursive inequality to upper bound the suboptimality at iteration $k+1$ in terms of the suboptimality in the previous $aK+1$ iterations (where $a$ is a positive integer that measures the length of history considered) and an additive remainder term. We will later show in Corollary \ref{corr1} that this remainder term can also be upper bounded in terms of the suboptimality observed in previous iterations.

\begin{thm}\label{thm1}
Suppose that Assumptions 1 and 2 hold. Then, the PIAG algorithm with step size $0<\eta\leq\frac{1}{L(K+1)}$ yields the following recursion
\begin{equation}\label{eq:rec}
  \left( 1 + \eta \frac{\mu}{8} \right) F_{k+1} \leq \left( 1- \sum_{i=1}^{a-1} \epsilon_i \right) F_k + \sum_{i=1}^{a-1} \epsilon_i \, F_{k-iK} + \eta^2 \frac{KL}{2} \, \epsilon_{a-1} \sum_{j=k-aK}^{k-1} \ndj^2,
\end{equation}
for any $k \geq aK+1$, where $a\geq2$ is an arbitrary constant and $\epsilon_i \triangleq 2\eta L \left( \eta KL \right)^{i-1}$.
\end{thm}

\begin{proof}
We use induction on the constant $a$ to prove this theorem. For $a=2$, the recursion can be obtained as follows. Using Lemma \ref{prop2} in Lemma \ref{prop1} and rearranging terms, we get
\begin{equation}
  \left( 1 + \eta \frac{\mu}{8} \right) F_{k+1} \leq F_k + \eta^2 L \sum_{j=k-K}^{k-1}  \ndj^2. \label{eq:t1}
\end{equation}
Rearranging terms in Lemma \ref{prop1}, we obtain
\begin{equation}
  \norm{d_j}^2 \leq \frac{2}{\eta} \left( F_j - F_{j+1} \right) + \eta L \sum_{i=j-K}^{j-1} \ndi^2. \label{eq:t2}
\end{equation}
Putting \eqref{eq:t2} back in \eqref{eq:t1}, we get
\begin{align}
  \left( 1 + \eta \frac{\mu}{8} \right) F_{k+1} & \leq F_k + \eta^2 L  \sum_{j=k-K}^{k-1} \left( \frac{2}{\eta} \left( F_j - F_{j+1} \right) + \eta L \sum_{i=j-K}^{j-1} \ndi^2 \right) \nn\\
    & \leq F_k + 2 \eta L \left( F_{k-K} - F_k \right) + \eta^3 L^2 \sum_{j=k-K}^{k-1} \sum_{i=j-K}^{j-1} \ndi^2  \nn\\
    & \leq F_k + 2 \eta L \left( F_{k-K} - F_k \right) + \eta^3 K L^2 \sum_{j=k-2K}^{k-1} \ndi^2 \label{eq:t3}
\end{align}
Since $\epsilon_1 = 2\eta L$, then \eqref{eq:t3} can be rewritten as follows
\begin{equation}
  \left( 1 + \eta \frac{\mu}{8} \right) F_{k+1} \leq (1-\epsilon_1) F_k + \epsilon_1 F_{(k-K)_+} + \eta^2 \frac{KL}{2} \, \epsilon_1 \sum_{j=k-2K}^{k-1} \ndj^2, \label{eq:t3_2}
\end{equation}
showing \eqref{eq:rec} for $a=2$. As a part of the induction procedure, we then assume that \eqref{eq:rec} holds for some arbitrary $a\geq2$, which amounts to
\begin{equation}
  \left( 1 + \eta \frac{\mu}{8} \right) F_{k+1} \leq \left( 1- \sum_{i=1}^{a-1} \epsilon_i \right) F_k + \sum_{i=1}^{a-1} \epsilon_i \, F_{k-iK} +  \eta^2 \frac{KL}{2} \, \epsilon_{a-1} \sum_{j=k-aK}^{k-1} \ndj^2. \nn
\end{equation}
Using \eqref{eq:t2} in the above inequality, we obtain
\begin{align}
  \left( 1 + \eta \frac{\mu}{8} \right) F_{k+1} & \leq \left( 1- \sum_{i=1}^{a-1} \epsilon_i \right) F_k + \sum_{i=1}^{a-1} \epsilon_i \, F_{k-iK} + \eta^2 \frac{KL}{2} \, \epsilon_{a-1} \nn\\
    & \hspace{3cm} \times \sum_{j=k-aK}^{k-1} \left( \frac{2}{\eta} \left( F_j - F_{j+1} \right) + \eta L \sum_{i=j-K}^{j-1} \ndi^2 \right)  \nn\\
    & \hspace{-1.5cm} = \left( 1- \sum_{i=1}^{a-1} \epsilon_i \right) F_k + \sum_{i=1}^{a-1} \epsilon_i \, F_{k-iK} + \epsilon_a \left( F_{(k-aK)_+} - F_k \right) + \eta^2 \frac{L}{2} \, \epsilon_a \sum_{j=k-aK}^{k-1} \sum_{i=j-K}^{j-1} \ndj^2 \nn\\
    & \hspace{-1.5cm} \leq \left( 1- \sum_{i=1}^a \epsilon_i \right) F_k + \sum_{i=1}^a \epsilon_i \, F_{k-iK} + \eta^2 \frac{KL}{2} \, \epsilon_a \sum_{j=k-(a+1)K}^{k-1} \ndj^2. \nn
\end{align}
Therefore, \eqref{eq:rec} holds for $a+1$ as well, which concludes the proof of Theorem \ref{thm1}.
\end{proof}

\begin{corr}\label{corr1}
Suppose that Assumptions 1 and 2 hold. Then, the PIAG algorithm with step size $0<\eta\leq\frac{1}{L(K+1)}$ yields the following recursion
\begin{equation}\label{eq:lyap_rec}
  \left( 1 + \eta \frac{\mu}{8} \right) F_{k+1} \leq \left( 1- \sum_{i=1}^{a-1} \epsilon_i \right) F_k + \sum_{i=1}^{a-1} \epsilon_i \, F_{k-iK} + 4KQ \, \epsilon_{a-1} \sum_{j=k-aK}^k F_j,
\end{equation}
for any integer $a\geq2$ and $k \geq aK+1$, where $\epsilon_i \triangleq 2\eta L \left( \eta KL \right)^{i-1}$.
\end{corr}

\begin{proof}
Using Assumption \ref{asmp:conv}, we obtain
\begin{align}
  \eta^2 \ndj^2 & = \norm{x_{j+1}-x_j}^2 \nn\\
    & \leq \left( \norm{x_{j+1}-x^*} + \norm{x_j-x^*} \right)^2 \nn\\
    & \leq 2 \left( \norm{x_{j+1}-x^*}^2 + \norm{x_j-x^*}^2 \right) \nn\\
    & \leq \frac{4}{\mu} \left( F_{j+1} + F_j \right), \label{eq:t5}
\end{align}
where the first inequality follows from the triangle inequality, the second inequality follows from the arithmetic-geometric mean inequality, and the last inequality follows from the $\mu$-strong convexity of the cost function. Using \eqref{eq:t5} in \eqref{eq:rec} leads to
\begin{align}
  \left( 1 + \eta \frac{8}{81Q^2L} \right) F_{k+1} & \leq \left( 1- \sum_{i=1}^{a-1} \epsilon_i \right) F_k + \sum_{i=1}^{a-1} \epsilon_i \, F_{k-iK} + 2 K \frac{L}{\mu} \, \epsilon_{a-1} \sum_{j=k-aK}^{k-1} \left( F_{j+1} + F_j \right) \nn\\
    & = \left( 1- \sum_{i=1}^{a-1} \epsilon_i \right) F_k + \sum_{i=1}^{a-1} \epsilon_i \, F_{k-iK} + 4 KQ \, \epsilon_{a-1} \sum_{j=k-aK}^k F_j, \nn
\end{align}
which completes the proof.
\end{proof}

Before presenting the main result of the paper, we first introduce the following lemma, which was presented in \cite[Lemma 3]{feyzmahdavian2014delayed} in a slightly different form.

\begin{lem}\label{mlsp}
Let $\{Z_k\}$ be a sequence of non-negative real numbers satisfying
\begin{equation}
  Z_{k+1} \leq p \, Z_k + \sum_{j=k-A}^k q_j Z_j, \nn
\end{equation}
for any $k\geq0$ for some non-negative constants $p$, $q_j$ and $A$. If $r \triangleq p+\sum_{j=k-A}^k q_j < 1$ holds, then
\begin{equation}
  Z_k \leq r^{\frac{k+1}{A+1}-1} \left( \max_{0 \leq j \leq A} Z_j \right),
\end{equation}
for any $k\geq A+1$.
\end{lem}

We next present the main theorem of the paper, which characterizes the linear convergence rate of the PIAG algorithm.

\begin{thm}\label{thm2}
Suppose that Assumptions 1 and 2 hold. Then, for any integer $a\geq3$, the PIAG algorithm with step size $0<\eta\leq\frac{1}{3L(K+1)} \overline{\eta}_a$ where
\begin{equation}
  \overline{\eta}_a \triangleq \left( \frac{1}{144(aK+1)KQ^2} \right)^{\frac{1}{a-2}}, \nn
\end{equation}
is linearly convergent satisfying
\begin{equation}
  F_k \leq \left( 1 - \eta \frac{\mu}{18} \right)^{\frac{k+1}{aK+1}-1} \left( \max_{0 \leq j \leq aK} F_j \right), \label{eq:thm2}
\end{equation}
for any $k \geq aK+1$.
\end{thm}

This theorem implies that the exact linear convergence rate depends on the parameter $a$, which measures the length of the history considered in Theorem \ref{thm1}. We will show that increasing $a$ will allow us to establish a linear convergence rate with better condition number $Q$ and aggregation history $K$ dependence. In particular, as $a\to\infty$, we have $\overline{\eta}_a\to1$ and we can pick a step size of $\eta=\frac{1}{3L(K+1)}$, which yields a linear convergence rate of $1-\order(1/(QK^2))$. However, this rate in \eqref{eq:thm2} is only achievable after the first $aK+1$ iterations. Therefore, picking larger step sizes yields a faster convergence rate but at the expense of a larger number of iterations to achieve that rate. We address this tradeoff in Corollary \ref{corr2} by showing the iteration complexity of the PIAG algorithm to achieve an $\epsilon$-optimal solution.

\begin{proof}
We will apply Lemma \ref{mlsp} to Corollary \ref{corr1} with $Z_k=F_k$, for which we require
\begin{itemize}
  \item [$(i)$] $0<\eta\leq\frac{1}{L(K+1)}$,
  \item [$(ii)$] $\epsilon_i \geq 0$ for all $i=1,\dots,a-1$,
  \item [$(iii)$] $1- \sum_{i=1}^{a-1} \epsilon_i \geq 0$, and
  \item [$(iv)$] $1 + \eta \frac{\mu}{8} > 1 + 4KQ (aK+1) \epsilon_{a-1}$.
\end{itemize}
The first and second conditions $(i)-(ii)$ are trivially satisfied as $0<\eta\leq\frac{1}{3L(K+1)} \overline{\eta}_a$ with $0<\overline{\eta}_a<1$, for any $a\in\Real$. In order to show that the third condition $(iii)$ holds, we use the definition of $\epsilon_i$, which implies
\begin{equation}
  \sum_{i=1}^{a-1} \epsilon_i = 2\eta L \sum_{i=1}^{a-1} \left( \eta KL \right)^{i-1} < 2\eta L \sum_{i=0}^\infty \left( \eta KL  \right)^i = 2\eta L \left( \frac{1}{1-\eta KL} \right), \label{eq:slem1}
\end{equation}
where the last equality follows since $\eta KL \leq \frac{K}{3(K+1)} \overline{\eta}_a <\frac{1}{3}$ as $\eta\leq\frac{1}{3L(K+1)} \overline{\eta}_a$ and $\overline{\eta}_a<1$. Then, we can upper bound the right hand side of \eqref{eq:slem1} as follows
\begin{equation}
  \sum_{i=1}^{a-1} \epsilon_i < 2 \eta L \left( \frac{1}{1-\frac{1}{3}} \right) = 3 \eta L \leq 1, \label{eq:slem2}
\end{equation}
where the inequalities follow since $\eta \leq \frac{1}{3L(K+1)} \overline{\eta}_a$ with $K\geq0$ and $\overline{\eta}_a<1$. The inequality \eqref{eq:slem2} implies $1- \sum_{i=1}^{a-1} \epsilon_i > 0$, therefore the third condition $(iii)$ holds as well. Using the definition of $\epsilon_{a-1}$, the fourth condition $(iv)$ can be written as follows
\begin{equation}\label{eq:const1}
  1 + \eta \frac{\mu}{8} > 1 + 8(aK+1)Q \left( \eta KL  \right)^{a-1},
\end{equation}
which can be equivalently expressed as
\begin{equation}\label{eq:const3}
  \eta < \frac{1}{KL} \left( \frac{1}{64(aK+1)KQ^2} \right)^{\frac{1}{a-2}},
\end{equation}
for any $a\geq3$ and $K>0$. We can trivially observe that for any $0<\eta\leq\frac{1}{3L(K+1)} \overline{\eta}_a$, the above inequality is satisfied, hence the fourth condition $(iv)$ holds as well.

Therefore, we can apply Lemma \ref{mlsp} to Corollary \ref{corr1}. This yields that for any step size $0<\eta\leq\frac{1}{3L(K+1)} \overline{\eta}_a$ with $a\geq3$, we obtain a global linear convergence
\begin{equation}
  F_k \leq \kappa^{\frac{k+1}{aK+1}-1} \max_{0\leq j\leq aK} F_j, \label{eq:final1}
\end{equation}
where
\begin{align}
  \kappa & = \frac{1 + 8(aK+1)Q \left( \eta KL  \right)^{a-1}}{1 + \eta \frac{\mu}{8}} \nn\\
    & = 1 - \eta \frac{\frac{\mu}{8} - 8(aK+1)KQL \left( \eta KL \right)^{a-2}}{1 + \eta \frac{\mu}{8}}. \nn
\end{align}
As $1<1 + \eta \frac{\mu}{8}\leq\frac{9}{8}$, we then have
\begin{equation}
  \kappa \leq 1 - \eta \frac{\mu}{9} + 8\eta(aK+1)KQL \left( \eta KL \right)^{a-2}. \label{eq:pot1}
\end{equation}
Noting that $\eta\leq\frac{1}{3L(K+1)} \overline{\eta}_a < \frac{1}{LK} \overline{\eta}_a$, we obtain $\left( \eta KL \right)^{a-2} < (\overline{\eta}_a)^{a-2} = \frac{1}{144(aK+1)KQ^2}$. Using this inequality in \eqref{eq:pot1}, we obtain
\begin{align}
  \kappa & < 1 - \eta \frac{\mu}{9} + 8\eta(aK+1)KQL \frac{1}{144(aK+1)KQ^2} \nn\\
    & \leq 1 - \eta \frac{\mu}{9} + \eta \frac{\mu}{18} \nn\\
    & \leq 1 - \eta \frac{\mu}{18}, \label{eq:const5}
\end{align}
which concludes the proof of Theorem \ref{thm2}.
\end{proof}

We next introduce the following corollary, which highlights the main result of the paper. This corollary indicates that using a step size of $\order(1/(KL\log(QK)))$, the PIAG algorithm is guaranteed to return an $\epsilon$-optimal solution after $\order(QK^2\log^2(QK)\log(1/\epsilon))$ iterations, or equivalently after $\torder(QK^2\log(1/\epsilon))$ iterations, where the tilde is used to hide the logarithmic terms in $Q$ and $K$.

\begin{corr}\label{corr2}
Suppose that Assumptions 1 and 2 hold. Then, the PIAG algorithm in \eqref{eq:updaterule} with step size $\eta = \frac{1}{3L(K+1)}\widetilde{\eta}_a$, where $\widetilde{\eta}_a \triangleq \frac{1}{a} \left( \frac{1}{12(K+1)Q} \right)^\frac{2}{a-2}$ and $a=\lceil\log(12(K+1)Q)\rceil+2$, is guaranteed to return an $\epsilon$-optimal solution after
\begin{equation}
  k \geq Ma^2(K+1)^2Q \log(c/\epsilon) + aK \label{eq:iteration_comp}
\end{equation}
iterations, where $M\triangleq54e^2$ and $c = \max_{0 \leq j \leq aK} F_j$ denotes the initial suboptimality in the function values.
\end{corr}

\begin{proof}
We begin the proof of this corollary by showing that $\widetilde{\eta}_a \leq \overline{\eta}_a$, i.e., $\eta = \frac{1}{3L(K+1)}\widetilde{\eta}_a$ is a valid step size satisfying \eqref{eq:thm2} of Theorem \ref{thm2}. Clearly, we have $(aK+1)K = a(K+\frac{1}{a})K \leq a (K+1)^2$ as $a\geq3$. Therefore, we obtain
\begin{equation}
  \overline{\eta}_a = \left( \frac{1}{144(aK+1)KQ^2} \right)^{\frac{1}{a-2}} \geq \left( \frac{1}{a} \right)^{\frac{1}{a-2}} \left( \frac{1}{12(K+1)Q} \right)^\frac{2}{a-2}.
\end{equation}
As $a\geq3$, we also have $\left( \frac{1}{a} \right)^{\frac{1}{a-2}} \geq \frac{1}{a}$, which indicates that $\overline{\eta}_a \geq \widetilde{\eta}_a$. Therefore, we can apply Theorem \ref{thm2} with step size $\eta = \frac{1}{3L(K+1)} \widetilde{\eta}_a$, which results in the following linear convergence
\begin{equation}
  F_k \leq c \left( 1 - \eta \frac{\mu}{18} \right)^{\frac{k+1}{aK+1}-1}, \label{eq:poc1}
\end{equation}
for any $k\geq aK+1$. Using the inequality $(1-x)^\gamma \leq 1-\gamma x$ for any $\gamma,x\in[0,1]$ in \eqref{eq:poc1} and noting that $a(K+1)\geq aK+1$, we obtain
\begin{equation}
  F_k \leq c \left( 1 - \eta \frac{\mu}{18a(K+1)} \right)^{k-aK}. \nn
\end{equation}
Taking the logarithm of both sides yields
\begin{align}
  \log(F_k) & \leq \log(c) + (k-aK) \log \left( 1 - \eta \frac{\mu}{18a(K+1)} \right) \nn\\
    & \leq \log(c) - (k-aK) \eta \frac{\mu}{18a(K+1)}, \label{eq:corr1}
\end{align}
where the last line follows since $\log(1+x)\leq x$ for any $x>-1$. Therefore, in order to achieve an $\epsilon$-optimal solution, the right-hand side of \eqref{eq:corr1} should be upper bounded by $\log(\epsilon)$, which implies
\begin{align}
  k & \geq \frac{18a(K+1)}{\eta\mu} \log(c/\epsilon) + aK \nn\\
    & =  \frac{54a(K+1)^2Q}{\widetilde{\eta}_a} \log(c/\epsilon) + aK, \label{eq:corr2}
\end{align}
where the equality follows since $\eta = \frac{1}{3L(K+1)} \widetilde{\eta}_a$. Picking $a=\lceil\log(12(K+1)Q)\rceil+2$, we can lower bound $\widetilde{\eta}_a$ as follows
\begin{align}
  \log(\widetilde{\eta}_a) & = \log \left( \frac{1}{a} \right) + \frac{2}{a-2} \log \left( \frac{1}{12(K+1)Q} \right) \nn\\
    & = - \log(a) - \frac{2}{\lceil\log(12(K+1)Q)\rceil} \log(12(K+1)Q) \nn\\
    & \geq - \log(a) - 2, \nn
\end{align}
which yields $\widetilde{\eta}_a \geq 1/(ae^2)$. Using this result in \eqref{eq:corr2}, we conclude that the PIAG algorithm is guaranteed to return an $\epsilon$-optimal solution after $k \geq 54e^2a^2(K+1)^2Q \log(c/\epsilon) + aK$ iterations.
\end{proof}

\section{Concluding Remarks}\label{sec:conclusion}
In this paper, we studied the PIAG method for additive composite optimization problems of the form \eqref{eq:goal}. We showed the first {linear convergence rate} result for the PIAG method and provided explicit convergence rate estimates that highlight the dependence on the condition number of the problem and the size of the window $K$ over which outdated component gradients are evaluated (under the assumptions that $f(x)$ is strongly convex and each $f_i(x)$ is smooth with Lipschitz gradients). Our results hold for any deterministic order (in processing the component functions) in contrast to the existing work on stochastic variants of our algorithm, which presents convergence results in expectation.

\bibliographystyle{plain}
\bibliography{piag_ref}

\end{document}